\documentclass[10pt, reqno]{amsart}
\usepackage{amsmath}
\usepackage{amsfonts}
\usepackage{amssymb}
\usepackage{amsthm}
\usepackage[utf8]{inputenc}
\usepackage{bm}
\usepackage{bbm}
\usepackage{graphicx}
\usepackage[dvipsnames]{xcolor}
\usepackage{hyperref}
\usepackage{url}
\usepackage{subcaption}
\usepackage{mathdots}
\usepackage[sc]{mathpazo}
\usepackage[T1]{fontenc} 
\usepackage{mathtools}

\usepackage{enumitem}

\mathtoolsset{showonlyrefs=true}

\usepackage[font=small,labelfont=bf]{caption}

\theoremstyle{plain}
\newtheorem{theorem}{Theorem}
\newtheorem{proposition}[theorem]{Proposition}
\newtheorem{lemma}[theorem]{Lemma}

\theoremstyle{definition}

\newtheorem{remark}[theorem]{Remark}

\numberwithin{equation}{section}

\newcommand{\C}{\mathbb{C}}
\newcommand{\F}{\mathcal{F}}
\newcommand{\N}{\mathbb{N}}
\newcommand{\R}{\mathbb{R}}
\newcommand{\E}{\mathbb{E}}
\newcommand{\M}{\mathrm{M}}

\renewcommand{\P}{\mathbb{P}}
\newcommand{\h}{\mathrm{H}}
\newcommand{\LL}{\boldsymbol{\lambda}}
\newcommand{\X}{\vec{X}}

\newcommand{\tr}{\operatorname{tr}}
\newcommand{\diag}{\operatorname{diag}}

\renewcommand{\vec}[1]{\mathbf{#1}}
\newcommand{\norm}[1]{\| #1 \|}

\newcommand{\bvec}[1]{\boldsymbol{#1}}
\newcommand{\TT}{\mathrm{T}}

\let\oldenumerate=\enumerate
	\def\enumerate{
	\oldenumerate
	\setlength{\itemsep}{5pt}
	}
\let\olditemize=\itemize
	\def\itemize{
	\olditemize
	\setlength{\itemsep}{5pt}
	}

\graphicspath{ {./images/} }

\begin{document}

\title[Norms on complex matrices induced by random vectors II]{Norms on complex matrices induced by random vectors II: extension of weakly unitarily invariant norms}
\author[\'A. Ch\'avez]{\'Angel Ch\'avez}
\address{Mathematics Department, Regis University, 3333 Regis Blvd., Denver, CO 80221} 
	\email{chave360@regis.edu}

\author[S.R.~Garcia]{Stephan Ramon Garcia}
\email{stephan.garcia@pomona.edu}
\address{Department of Mathematics and Statistics, Pomona College, 610 N. College Ave. Claremont, CA 91711} 
\urladdr{\url{https://stephangarcia.sites.pomona.edu/}}

\author[J. Hurley]{Jackson Hurley}
\email{jacksonwhurley@gmail.com}

\thanks{SRG partially supported by NSF grant DMS-2054002.}

\begin{abstract}
We improve and expand in two directions the theory of norms on complex matrices induced by random vectors. We first provide a simple proof of the classification of weakly unitarily invariant norms on the Hermitian matrices.  We use this to extend the main theorem in \cite{Ours} from exponent $d\geq 2$ to $d \geq 1$.  Our proofs are much simpler than the originals: they do not require Lewis' framework for group invariance in convex matrix analysis.  This clarification puts the entire theory on simpler foundations while extending its range of applicability.
\end{abstract}

\maketitle

\section{Introduction}

 A norm $\norm{\cdot}$ on $\M_n$, the space of $n\times n$ complex matrices, is \emph{unitarily invariant} 
if $\norm{UAV}=\norm{A}$ for all $A\in \M_n$ and unitary $U,V \in \M_n$.  A norm on $\R^n$ which is invariant under entrywise sign changes and permutations is a \emph{symmetric gauge function}. A theorem of von Neumann asserts that any unitarily invariant norm on $\M_n$ is a symmetric gauge function applied to the singular values \cite[Thm.~7.4.7.2]{Horn}. 
For example, the Schatten norms are unitarily invariant and defined for $d\geq 1$ by
\begin{equation*}
    \norm{A}_{S_d}=\big( |\sigma_1|^d+|\sigma_2|^d+\cdots+ |\sigma_n|^d\big)^{1/d},
\end{equation*} 
in which $\sigma_1 \geq \sigma_2 \geq  \cdots \geq \sigma_n \geq 0$ are the singular values of $A\in \M_n$. 

A norm $\norm{\cdot}$ on the $\R$-vector space $\h_n$ of $n\times n$ complex Hermitian matrices is \emph{weakly unitarily invariant} if 
$\norm{U^*AU}=\norm{A}$ for all $A\in \h_n$ and unitary $U \in \M_n$. 
For example, the numerical radius 
\begin{equation*}
r(A) = \sup_{  \vec{x} \in \C^n \backslash \{\vec{0}\}} \frac{ \langle A \vec{x}, \vec{x} \rangle}{\langle \vec{x} , \vec{x} \rangle}
\end{equation*}
is a weakly unitarily invariant norm on $\h_n$ \cite{LiUnitary}.
Lewis proved that any weakly unitarily invariant norm on $\h_n$ is a symmetric vector norm applied to the eigenvalues \cite[Sec.~8]{LewisGroupInvariance}.

Our first result is a short proof of Lewis' theorem that avoids his theory of group invariance in convex matrix analysis \cite{LewisGroupInvariance}, the wonderful but complicated framework that underpins \cite{First,Ours}.  Our new approach
uses more standard techniques, such as Birkhoff's theorem on doubly stochastic matrices \cite{Birkhoff}.

\begin{theorem}\label{Theorem:Main}
    A norm $\norm{\cdot}$ on $\h_n$ is  weakly unitarily invariant if and only if
    there is a symmetric norm $f:\R^n\to\R$ such that $\norm{A}=f( \lambda_1, \lambda_2, \ldots, \lambda_n)$
    for all $A\in \h_n$. Here, $\lambda_1 \geq \lambda_2 \geq \cdots \geq \lambda_n$ are the eigenvalues of $A$.
\end{theorem}

The \emph{random-vector norms} of the next theorem are weakly unitarily invariant norms on $\h_n$ that extend to weakly unitarily invariant norms on $\M_n$; see Theorem \ref{Theorem:Big} below.
They appeared in \cite{Ours} and they generalize the complete homogeneous symmetric polynomial norms of \cite[Thm.~1]{First}. 
The original proof of  \cite[Thm.~1.1 (a)]{Ours} requires $d \geq 2$ and relies heavily on Lewis' framework for group invariance in convex matrix analysis \cite{LewisGroupInvariance}. However, Theorem \ref{Theorem:Second} now follows directly from Theorem \ref{Theorem:Main}. Moreover, Theorem \ref{Theorem:Second} generalizes  \cite[Thm.~1.1 (a)]{Ours} to the case $d\geq 1$. 

\begin{theorem}\label{Theorem:Second}
    Let $d\geq 1$ be real and $\vec{X}$ be an \emph{iid random vector} in $\R^n$, that is, the entries of
    $\vec{X}=(X_1,X_2, \ldots, X_n)$ are nondegenerate iid random variables. Then
    \begin{equation}\label{eq:NormHermitian}
        \norm{A}_{\vec{X},d} = \left( \frac{\E |\langle\vec{X}, \bm{\lambda}\rangle|^d}
                {\Gamma(d+1)}\right)^{1/d}
    \end{equation} 
    is a weakly unitarily invariant norm on $\h_n$. Here, $\Gamma(\cdot)$ denotes the gamma function and 
    $\boldsymbol{\lambda}=(\lambda_1,\lambda_2, \ldots, \lambda_n)$ denotes the vector of eigenvalues 
    $\lambda_1 \geq \lambda_2 \geq \cdots \geq \lambda_n$ of $A$.
    Moreover, if the entries of $\vec{X}$ each have at least $m$ moments, then for all $A\in\h_n$ the function $f:[1,m] \to \R$ defined by 
    $f(d) =\norm{A}_{\vec{X},d}$ is continuous. 
\end{theorem}

The simplified proof of Theorem \ref{Theorem:Main} and the extension of Theorem \ref{Theorem:Second} from $d\geq 2$ to 
$d \geq 1$ permits the main results of \cite{Ours}, restated below as Theorem \ref{Theorem:Big}, to rest on simpler foundations 
while enjoying a wider range of applicability.
The many perspectives offered in Theorem \ref{Theorem:Big} below explain the normalization in \eqref{eq:NormHermitian}. 

 \begin{theorem}\label{Theorem:Big}
    Let $\X=(X_1, X_2, \ldots, X_n)$, in which
    $X_1, X_2, \ldots, X_n \in L^d(\Omega,\F,\P)$ are nondegenerate iid random variables.
    Let   $\boldsymbol{\lambda}=(\lambda_1,\lambda_2, \ldots, \lambda_n)$ denote the vector of eigenvalues 
    $\lambda_1 \geq \lambda_2 \geq \cdots \geq \lambda_n$ of $A \in \h_n$.    
    \addtolength{\leftmargini}{-18pt}
    \begin{enumerate}\addtolength{\itemsep}{0pt}
    \item For real $d\geq 1$, $\norm{A}_{\X,d}= \bigg(\dfrac{  \E |\langle \X, \LL\rangle|^d}{\Gamma(d+1)} \bigg)^{1/d}$ is a norm on $\h_n$.
           \hfill (now by Theorem \ref{Theorem:Second}).
    
    \item If the $X_i$ admit a moment generating function $M(t) = \E [e^{tX}] = \sum_{k=0}^{\infty} \E [X^k] \frac{t^k}{k!}$ and $d \geq 2$ is an even integer, then
    $\norm{A}_{\X,d}^d$ is the coefficient of $t^d$ in $M_{\Lambda}(t)$ for all $A \in \h_n$,
    in which $M_{\Lambda}(t) = \prod_{i=1}^n M(\lambda_i t)$ is the moment generating
    function for the random variable 
    $\Lambda =\langle \X, \LL(A) \rangle=\lambda_1X_1+\lambda_2X_2+\cdots +\lambda_n X_n$.
    In particular, $\norm{A}_{\X,d}$ is a positive definite, homogeneous,
    symmetric polynomial in the eigenvalues of $A$.
    
    \item Let $d\geq 2$ be an even integer. If the first $d$ moments of $X_i$ exist, then
    \begin{equation*}
    \norm{A}_{\X,d}^d
    = \frac{1}{d!} B_{d}(\kappa_1\tr A, \kappa_2\tr A^2, \ldots, \kappa_d\tr A^d) \label{eq:MainBell} 
    = \sum_{\bvec{\pi}\vdash d}\frac{\kappa_{\bvec{\pi}}p_{\bvec{\pi}} (\bvec{\lambda})}{y_{\bvec{\pi}}}
    \quad \text{for $A \in \h_n$}, 
    \end{equation*}
    in which:
    \begin{enumerate}
    	\item $\bvec{\pi}=(\pi_1, \pi_2, \ldots, \pi_r) \in \N^r$ is a \emph{partition} of $d$; that is, $\pi_1 \geq \pi_2 \geq \cdots \geq \pi_r$ and
            $\pi_1+ \pi_2 + \cdots + \pi_r = d$ \cite[Sec.~1.7]{StanleyBook1}; we denote this $\bvec{\pi} \vdash d$. 

    	\item $p_{\bvec{\pi}}(x_1, x_2, \ldots, x_n)=p_{\pi_1}p_{\pi_2}\cdots p_{\pi_r}$,
in which $p_k(x_1,x_2, \ldots, x_n)=x_1^k+x_2^k+\cdots +x_n^k$ is a \emph{power-sum symmetric polynomial}.

    	\item $B_d$ is a complete Bell polynomial, defined by $\sum_{\ell=0}^{\infty} B_{\ell}(x_1, x_2, \ldots, x_{\ell}) \frac{t^{\ell}}{\ell !}\break=\exp( \sum_{j=1}^{\infty} x_j \frac{t^j}{j!})$ \cite[Sec.~II]{Bell}.
	
    	\item The \emph{cumulant}s $\kappa_1, \kappa_2, \ldots, \kappa_d$ are defined by the recursion
            $\mu_r=\sum_{\ell=0}^{r-1}{r-1\choose \ell} \mu_{\ell}\kappa_{r-\ell}$ for $1 \leq r \leq d$, in which
            $\mu_r = \E[X_1^r]$ is the $r$th moment of $X_1$ \cite[Sec.~9]{Billingsley}.
	
    	\item $\kappa_{\bvec{\pi}} = \kappa_{\pi_1} \kappa_{\pi_2} \cdots \kappa_{\pi_{r}}$ and
		$y_{\bvec{\pi}}=\prod_{i\geq 1}(i!)^{m_i}m_i!$, in which $m_i=m_i(\bvec{\pi})$ is the multiplicity of $i$ in $\bvec{\pi}$.
    \end{enumerate}

    \item For real $d\geq 1$, the function $\bvec{\lambda}(A) \mapsto \norm{A}_{\X,d}$ is Schur convex; that is, it respects majorization $\prec$;
    	see \eqref{eq:Majorization}.
    
    \item Let $d\geq 2$ be an even integer. Define $\TT_{\bvec{\pi}} : \M_{n}\to \R$
    by setting $\TT_{\bvec{\pi}}(Z)$ to be $1/{d\choose d/2}$ times the sum over the $\binom{d}{d/2}$ 
    possible locations to place $d/2$ adjoints ${}^*$ among the $d$ copies of $Z$ in
    $(\tr \underbrace{ZZ\cdots Z}_{\pi_1})
    (\tr \underbrace{ZZ\cdots Z}_{\pi_2})
    \cdots
    (\tr \underbrace{ZZ\cdots Z}_{\pi_r})$.
    Then
    \begin{equation}\label{eq:Extended}
    \norm{Z}_{\X,d}= \bigg( \sum_{\bvec{\pi} \,\vdash\, d} \frac{ \kappa_{\bvec{\pi}}\TT_{\bvec{\pi}}(Z)}{y_{\bvec{\pi}}}\bigg)^{1/d}
    \end{equation} 
    is a norm on $\M_n$ that restricts to the norm on $\h_n$ above.  In particular, $\norm{Z}_{\X,d}^d$ 
    is a positive definite trace polynomial in $Z$ and $Z^*$.
    \end{enumerate}
\end{theorem}

The paper is structured as follows. Section \ref{Section:Examples} provides several examples afforded by
the theorems above. Proofs of Theorems \ref{Theorem:Main} and \ref{Theorem:Second} appear in Sections \ref{Section:Proof1} and \ref{Section:Proof2}, respectively.  Section \ref{Section:End} concludes with some brief remarks.

\medskip\noindent\textbf{Acknowledgments.} We thank the referee for many helpful comments.

\section{Examples}\label{Section:Examples}
The norm $\norm{\cdot}_{\vec{X},d}$ defined in \eqref{eq:NormHermitian} is determined by its unit ball.
This provides one way to visualize the properties of random vector norms.
We consider a few examples here and refer the reader to \cite[Sec.~2]{Ours} for further examples and details.

\subsection{Normal Random Variables} 
Suppose $d\geq 2$ is an even integer and $\vec{X}$ is a random vector whose entries are independent normal random variables with mean $\mu$ and variance $\sigma^2$.  The example in \cite[(2.12)]{Ours} illustrates
\begin{equation*}
        \norm{A}_{\X,d}^d=\sum_{k=0}^{\frac{d}{2}}  \frac{\mu^{2k} (\tr A)^{2k}}{(2k)!} 
        \cdot  \frac{\sigma^{d-2k} \norm{A}_{\operatorname{F}}^{d-2k}}{2^{\frac{d}{2}-k} (\frac{d}{2}-k)!}
        \quad \text{for $A \in \h_n$},
\end{equation*}
in which $\norm{\cdot}_{\operatorname{F}}$ is the Frobenius norm.  For $d=2$ the extension to $\M_n$ guaranteed by 
Theorem \ref{Theorem:Big} is $\norm{Z}_{\X,2}^2= \tfrac{1}{2} \sigma^2 \tr(Z^*\!Z) + \tfrac{1}{2} \mu^2 (\tr Z^*)(\tr Z)$ \cite[p.~816]{Ours}.

Now let $n=2$.
If $\mu=0$, the restrictions of $\norm{\cdot}_{\vec{X},d}$ to $\R^2$ (whose elements are identified with diagonal matrices)
reproduce multiples of the Euclidean norm.
If $\mu\neq 0$, then the unit circles for $\norm{\cdot}_{\vec{X},d}$ are approximately elliptical;
see Figure \ref{Figure:Normal}.

\begin{figure}
\centering
\includegraphics[width = 0.475\textwidth]{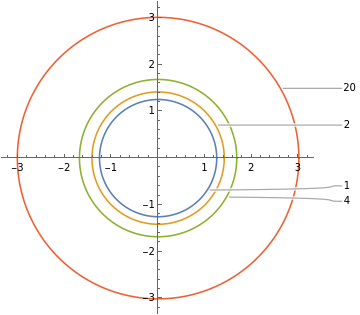}\hfill 
\includegraphics[width = 0.475\textwidth]{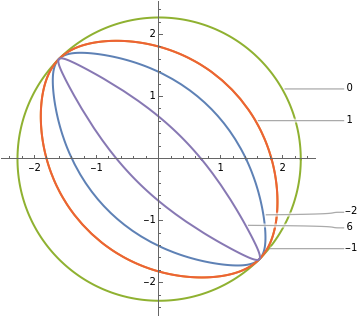}
\caption{(\textsc{Left}) Unit circles for $\|\cdot\|_{\vec{X},d}$ with $d=1, 2, 4, 20$, in which $X_1$ and $X_2$ are standard normal random variables. (\textsc{Right}) Unit circles for $\|\cdot\|_{\vec{X},10}$, in which $X_1$ and $X_2$ are normal random variables with means $\mu=-2, -1, 0, 1, 6$ and variance $\sigma^2=1$.}
\label{Figure:Normal}
\end{figure}


\subsection{Standard Exponential Random Variables} 
If $d\geq 2$ is an even integer and $\vec{X}$ is a random vector whose entries are independent standard exponential random variables, then $\norm{A}_{\vec{X},d}^d$ equals the \emph{complete homogeneous symmetric polynomial} 
$h_d(\lambda_1, \lambda_2, \ldots, \lambda_n)=\sum_{1\leq k_1\leq\cdots\leq k_d\leq n} \lambda_{k_1}\lambda_{k_2}\cdots \lambda_{k_d}$
in the eigenvalues $\lambda_1, \lambda_2, \ldots, \lambda_n$ \cite{First}. 
For $d=4$ the extension to $\M_n$ guaranteed by Theorem \ref{Theorem:Big}
is \cite[(9)]{First}:
\begin{align*}
\norm{Z}_4^4 &= \frac{1}{24} \big(
(\tr Z)^2 \tr(Z^*)^2 + \tr(Z^*)^2 \tr(Z^2) + 4 \tr(Z) \tr(Z^*) \tr(Z^*Z)    \nonumber \\ 
& \qquad + 
 2 \tr(Z^*Z)^2 + (\tr Z)^2 \tr(Z^{*2}) + \tr(Z^2) \tr(Z^{*2}) +  4 \tr(Z^*) \tr(Z^*Z^2) \nonumber \\
  & \qquad + 4 \tr(Z) \tr(Z^{*2}Z) + 2 \tr(Z^*ZZ^*Z) + 4 \tr(Z^{*2}Z^2) \big). 
\end{align*}
The unit balls for these norms are illustrated in Figure \ref{Figure:Exponentials} (\textsc{Left}).

\begin{figure}
\centering
\includegraphics[width = 0.475\textwidth]{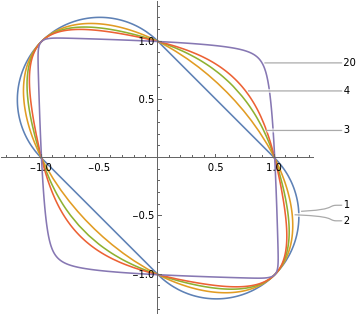}
\hfill\includegraphics[width = 0.475\textwidth]{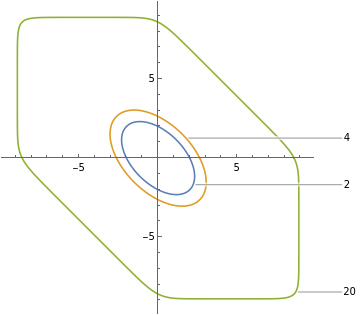}
\caption{(\textsc{Left}) Unit circles for $\|\cdot\|_{\vec{X},d}$ with $d=1, 2, 3, 4, 20$, in which $X_1$ and $X_2$ are standard exponentials. (\textsc{Right}) Unit circles for $\norm{\cdot}_{\vec{X},d}$ with $d=2, 4, 20$, in which $X_1$ and $X_2$ are Bernoulli with $q=0.5$.}
\label{Figure:Exponentials}
\end{figure}


\subsection{Bernoulli Random Variables} A \emph{Bernoulli} random variable is a discrete random variable $X$ defined by $\mathbb{P}(X=k)=q^k(1-q)^{1-k}$ for $k=0,1$ and $0<q<1$. Suppose $d$ is an even integer and $\vec{X}$ is a random vector whose entries are independental Bernoulli random variables with parameter $q$. 

\begin{remark}
An expression for $\norm{A}^d_{\vec{X},d}$ appears in \cite[Sec.~2.7]{Ours}. However, there is a missing multinomial coefficient. The correct expression for$\norm{A}^d_{\vec{X},d}$ is given by \begin{equation*}
    \norm{A}^d_{\vec{X},d}=\frac{1}{d!}\sum_{i_1+i_2+\cdots+i_n=d}
    \binom{d}{i_1,i_2,\ldots,i_n}q^{|I|} \lambda_1^{i_1}\lambda_2^{i_2}\cdots \lambda_n^{i_n},
\end{equation*} 
in which $|I|$ is the number of nonzero $i_k$; that is,
$I = \{ k : i_k \neq 0\}$. We thank the anonymous referee for pointing out the typo in \cite[Sec.~2.7]{Ours}.
Figures \ref{Figure:Exponentials} (\textsc{Right}) 
and \ref{Figure:Bernoulli} illustrate the unit balls for these norms in a variety of cases. 
\end{remark}

\begin{figure}
\centering
\includegraphics[width = 0.475\textwidth]{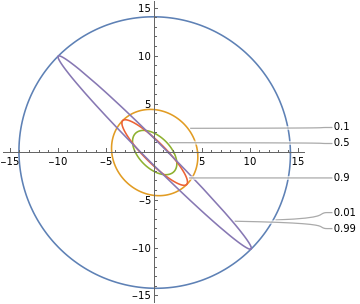}
\hfill
\includegraphics[width = 0.475\textwidth]{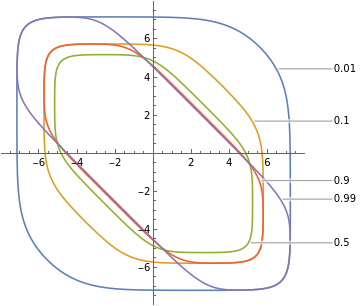}
\caption{Unit circles for $\|\cdot\|_{\vec{X},d}$, in which $X_1$ and $X_2$ are Bernoulli with varying parameter $q$ and with $d=2$ (\textsc{Left}) and $d=10$ (\textsc{Right}).}
\label{Figure:Bernoulli}
\end{figure}


\subsection{Pareto Random Variables} 
Suppose $\alpha, x_m>0$. A random variable $X$ distributed according to the probability density function
\begin{align*}
f_X(t)=
\begin{cases}
\dfrac{\alpha x_m^{\alpha}}{t^{\alpha+1}} & t\geq x_m,\\ 
0 & t<x_m,
\end{cases}
\end{align*} 
is a \emph{Pareto} random variable with parameters $\alpha$ and $x_m$. Suppose $\vec{X}$ is a random vector whose entries are Pareto random variables. Then $\norm{A}_{\vec{X},d}$ exists whenever $\alpha>d$ \cite[Sec.~2.10]{Ours}.

Suppose $d=2$ and $\vec{X}$ is a random vector whose entries are independent Pareto random variables with  $\alpha>2$ and $x_m=1$. If $n=2$, then
\begin{align*}
\norm{A}_{\vec{X},2}^2=\frac{\alpha}{2}\Big( \frac{\lambda_1^2}{\alpha-2}+\frac{2\alpha \lambda_1\lambda_2}{(\alpha-1)^2}+\frac{\lambda_2^2}{\alpha-2}     \Big).
\end{align*}
Figure \ref{Pareto124} (\textsc{Left}) illustrates the unit circles for $\norm{\cdot}_{\vec{X},2}$ with varying $\alpha$.
As $\alpha\to \infty$ the unit circles approach the parallel lines at $\lambda_2=\pm \sqrt{2}-\lambda_1$; that is, $|\tr A|^2 = 2$. Figure \ref{Pareto124} (\textsc{Right}) depicts the unit circles for $\norm{\cdot}_{\vec{X},d}$ with fixed $\alpha$ and varying $d$. 

\begin{figure}
\centering
\includegraphics[width = 0.45\textwidth]{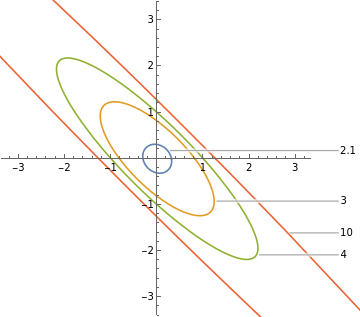}
\hfill
\includegraphics[width = 0.45\textwidth]{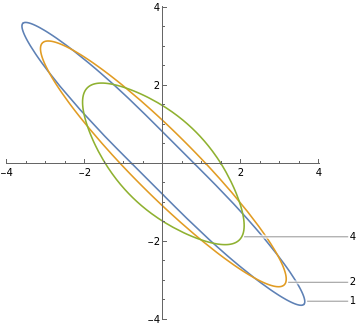}
\caption{(\textsc{Left}) Unit circles for $\norm{\cdot}_{\vec{X},2}$, in which $X_1$ and $X_2$ are independent Pareto random variables with $\alpha=2.1, 3, 4, 10$ and $x_m=1$. (\textsc{Right}) Unit circles for $\norm{\cdot}_{\vec{X},d}$, in which $X_1$ and $X_2$ are independent Pareto random variables with $\alpha=5$ and $p=1, 2, 4$. }
\label{Pareto124}
\end{figure}


\section{Proof of Theorem \ref{Theorem:Main}}\label{Section:Proof1}
The proof of Theorem \ref{Theorem:Main} follows from Propositions \ref{Proposition:OnlyIf} and \ref{Proposition:RntoHermitian} below.
       
\begin{proposition}\label{Proposition:OnlyIf}
If $\norm{\cdot}$ is a weakly unitarily invariant norm on $\h_n$, then there is a symmetric 
norm $f$  on $\mathbb{R}^n$ such that $\norm{A}=f(  \boldsymbol{\lambda}(A))$ for all $A\in \h_n$.
\end{proposition}
        
\begin{proof}
   Hermitian matrices are unitarily diagonalizable. Since $\norm{\cdot}$ is weakly unitarily invariant, $\norm{A}=\norm{D}$, in which $D$ is a diagonalization of $A$. Consequently, $\norm{A}$  must be a function in the eigenvalues of $A$. Moreover, any permuation of the entries in $D$ is obtained by conjugating $D$ by a permutation matrix, which is unitary. Therefore, $\norm{A}$ is a symmetric function in the eigenvalues of $A$. 
    In particular, $\norm{A}=f(  \boldsymbol{\lambda}(A) )$ for some symmetric function $f$.   Given 
    $\vec{a}=( a_1, a_2,\dots, a_n)\in \R^n$,  define the Hermitian matrix
    \begin{equation*}
     \diag{\vec{a}} = 
      \begin{bmatrix}
      a_1 & & &\\
      & a_2 & &\\
      & & \ddots & \\
      & & & a_n
     \end{bmatrix}.
    \end{equation*} 
    Then $\bm{\lambda}(\diag{\vec{a}}) = P\vec{a}$ for some permutation matrix $P$. Symmetry of $f$ implies
    \begin{equation*}
        f(\vec{a}) =f(P\vec{a}) = f\big(\bm{\lambda}(\diag{\vec{a}})\big) = \norm{\diag{\vec{a}}}.
    \end{equation*} 
    Consequently, $f$ inherits the properties of a norm on $\R^n$.
\end{proof}

Let $\widetilde{\vec{x}}=(\widetilde{x}_1,\widetilde{x}_2, \ldots, \widetilde{x}_n)$ denote the nondecreasing rearrangement of 
$\vec{x}= \break (x_1, x_2, \ldots, x_n)\in \R^n$. Then $\vec{y}$ \emph{majorizes} $\vec{x}$, denoted $\vec{x}\prec \vec{y}$, if
\begin{equation}\label{eq:Majorization}
\sum_{i=1}^n \widetilde{x}_i=\sum_{i=1}^n \widetilde{y}_i
\quad \text{and} \quad\sum_{i=1}^k \widetilde{x}_i\leq \sum_{i=1}^k \widetilde{y}_i \quad \text{for $1 \leq k \leq n-1$}.
\end{equation} 
Recall that a matrix with nonnegative entries is \emph{doubly stochastic} if each row and column sums to $1$.
The next result is due to Hardy, Littlewood and P\'olya \cite{Hardy}.

\begin{lemma}\label{Lemma:MajorStoch}
If $\vec{x}\prec\vec{y}$, then there exists a doubly stochastic matrix $D$ such that $\vec{y} = D \vec{x}$. 
\end{lemma}

The next lemma is Birkhoff's \cite{Birkhoff};  $n^2-n+1$ works in place of $n^2$ \cite[Thm.~8.7.2]{Horn}.

\begin{lemma}\label{stochasticisPerms}
If $D \in \M_n$ is doubly stochastic, then there exist permutation matrices $P_1,P_2,\ldots,P_{n^2} \in \M_n$
    and nonnegative numbers $c_1,c_2,\ldots,c_{n^2}$ satisfying $\sum_{i=1}^{n^2} c_i = 1$ such that
    $D = \sum_{i=1}^{n^2} c_i P_i$.
\end{lemma}

For each $A \in \h_n$, recall that  $\boldsymbol{\lambda}(A)=(\lambda_1(A),\lambda_2(A), \ldots, \lambda_n(A))$ denotes the vector of eigenvalues 
    $\lambda_1(A) \geq \lambda_2(A) \geq \cdots \geq \lambda_n(A)$.    We regard $\boldsymbol{\lambda}(A)$ as a column vector
    for purposes of matrix multiplication.

\begin{lemma}\label{Lemma:EigenPerms}
If $A, B\in \h_n$, then there exist permutation matrices $P_1,P_2,\ldots,P_{n^2} \in \M_n$ and 
$c_1,c_2,\ldots,c_{n^2}\geq 0$  such that 
\begin{align*}
\bm{\lambda}(A+B) = \sum_{i = 1}^{n^2} c_{i} P_i(\bm{\lambda}(A) + \bm{\lambda}(B))
\quad \text{and}\quad 
\sum_{i = 1}^{n^2} c_i = 1.
\end{align*}
\end{lemma}

\begin{proof}
The Ky Fan eigenvalue inequality \cite{KyFan} asserts that
\begin{equation}\label{eq:KyFan}
    \sum_{i=1}^{k} \lambda_i(A+B) \leq \sum_{i=1}^{k} \big(\lambda_i(A) + \lambda_i(B)\big)
    \quad \text{for all $1\leq k\leq n$}.
\end{equation} 
The sum of the eigenvalues of a matrix is its trace. Consequently,
\begin{align*} 
\sum_{i=1}^{n} \lambda_i(A+B) = \tr (A+B) = \tr  A + \tr  B = \sum_{i=1}^{n} \big(\lambda_i(A) + \lambda_i(B)\big),
\end{align*} 
so equality holds in \eqref{eq:KyFan} for $k=n$. Thus, $\bm{\lambda}(A+B) \prec \bm{\lambda}(A) + \bm{\lambda}(B)$. 
Lemma \ref{Lemma:MajorStoch} provides a doubly stochastic matrix $D$ such that 
$\bm{\lambda}(A+B) = D(\bm{\lambda}(A) + \bm{\lambda}(B))$. 
Lemma \ref{stochasticisPerms} provides the desired permutation matrices and nonnegative scalars.
\end{proof}

The following proposition completes the proof of Theorem \ref{Theorem:Main}.  

\begin{proposition}\label{Proposition:RntoHermitian}
    If $f$ is a symmetric norm on $\R^n$, then 
    $\norm{A}=f(\boldsymbol{\lambda}(A))$ defines a weakly unitarily invariant norm on $\h_n$.
\end{proposition}

\begin{proof}
    The function $\norm{A}=f(\boldsymbol{\lambda}(A))$ is symmetric in the eigenvalues of $A$, so it is
    weakly unitarily invariant. It remains to show that $\norm{\cdot}$ defines a norm on $\h_n$. 
    
    \medskip\noindent\textsc{Positive Definiteness}. A Hermitian matrix $A = 0$ if and only if $\bm{\lambda}(A) = 0$. Thus, 
    the positive definiteness of $f$ implies the positive definiteness of $\norm{\cdot}$. 
    
    \medskip\noindent\textsc{Homogeneity}. If $c\geq 0$, then $\bm{\lambda}(cA) = c\bm{\lambda}(A)$. If $c<0$, then
    \begin{equation*}
        \bm{\lambda}(cA) = c
        \begin{bmatrix}
        && 1  \\
        & \iddots & \\
        1 & &
        \end{bmatrix}
        \bm{\lambda}(A).
    \end{equation*} 
    Then the homogeneity and symmetry of $f$ imply that 
    \begin{equation*}
        \norm{cA} = f\big(\bm{\lambda}(cA)\big) = f\big(c\bm{\lambda}(A)\big) = |c| f\big(\bm{\lambda}(A)\big) = |c|\norm{A}.
    \end{equation*} 
    
    \medskip\noindent\textsc{Triangle Inequality}. 
    Suppose  that $A,B \in \h_n$. 
    Lemma \ref{Lemma:EigenPerms} ensures that there exist permutation matrices $P_1,P_2,\ldots,P_{n^2} \in \M_n$
    and nonnegative numbers $c_1,c_2,\ldots,c_{n^2}$ satisfying $\sum_{i=1}^{n^2} c_i = 1$ such that
    $D = \sum_{i=1}^{n^2} c_i P_i$.  Thus,
    \begin{equation*}
        \norm{A + B} = f\big(\bm{\lambda}(A+B)\big) = f\Big(\sum_{i = 1}^{n^2} c_{i} P_i\big(\bm{\lambda}(A) + \bm{\lambda}(B)\big)\Big).
    \end{equation*} 
    The triangle inequality and homogeneity of $f$ yield
    \begin{equation}\label{eq:Triangle}
        \norm{A+B}\leq \sum_{i=1}^{n^2} c_{i} f\big( P_i(\bm{\lambda}(A) + \bm{\lambda}(B))\big).
    \end{equation} 
    Since $f$ is permutation invariant and $\sum_{i = 1}^{n^2} c_i = 1$,
    \begin{equation*}
        \sum_{i=1}^{n^2} c_{i} f\big( P_i(\bm{\lambda}(A) + \bm{\lambda}(B))\big)
        = \sum_{i=1}^{n^2} c_{i} f\big(\bm{\lambda}(A) + \bm{\lambda}(B)\big)= f\big(\bm{\lambda}(A) + \bm{\lambda}(B)\big).
    \end{equation*} 
    Thus, the triangle inequality for $f$ and \eqref{eq:Triangle} yield
    \begin{equation*}
        \norm{A+B}
        \leq  f\big(\bm{\lambda}(A) + \bm{\lambda}(B)\big)
        \leq  f\big(\bm{\lambda}(A) \big)+ f\big(\bm{\lambda}(B)\big)
        =\norm{A}+\norm{B}. \qedhere
    \end{equation*}
\end{proof}

\section{Proof of Theorem \ref{Theorem:Second}}\label{Section:Proof2}

Let $\vec{X}$ be an iid random vector and define $f_{\vec{X},d}:\R^n\to \R$ by
\begin{equation}\label{eq:norm}
     f_{\vec{X},d}(\boldsymbol{\lambda}) = \left( \frac{\E |\langle\vec{X}, \bm{\lambda}\rangle|^d}{\Gamma(d+1)}\right)^{1/d}
     \quad \text{for $d \geq 1$}.
\end{equation}  
Since the entries of $\vec{X}$ are independent and identically distributed, $f_{\vec{X},d}$ is symmetric. 
In light of Theorem \ref{Theorem:Main}, it suffices to show that $f_{\vec{X},d}$ is a norm on $\R^n$;
the continuity remark at the end of Theorem \ref{Theorem:Second} is Proposition \ref{Proposition:Continuity} below.

\begin{proposition}
The function $f_{\vec{X},d}$ in \eqref{eq:norm} defines a norm on $\R^n$ for all $d\geq 1$.
\end{proposition}

\begin{proof} 
The proofs for homogeneity and the triangle inequality in \cite[Sec.~3.1]{Ours} are valid for $d\geq 1$. 
However, the proof for positive definiteness in \cite[Lemma~3.1]{Ours} requires $d\geq 2$. The proof below
holds for $d\geq 1$ and is simpler than the original.

\medskip\noindent\textsc{Positive Definiteness}.  
If $f_{\vec{X},d}(\boldsymbol{\lambda})=0$, then $\E|\langle \vec{X},\boldsymbol{\lambda}\rangle|^d=0$. 
The nonnegativity of $|\langle \vec{X},\boldsymbol{\lambda}\rangle|^d$ ensures that
\begin{equation}\label{eq:Independence}
    \lambda_1X_1+\lambda_2X_2+\cdots+\lambda_nX_n=0 
\end{equation} 
almost surely. Assume \eqref{eq:Independence} has a nontrivial solution $\boldsymbol{\lambda}$ with nonzero entries $\lambda_{i_1}, \lambda_{i_2}, \ldots, \lambda_{i_k}$. If $k=1$, then $X_{i_k}=0$ almost surely, which contradicts the nondegeneracy of our random variables. If $k>1$, then \eqref{eq:Independence} implies that
\begin{equation}\label{eq:Contradiction}
    X_{i_1}=a_{i_2}X_{i_2}+a_{i_3}X_{i_3}+\cdots +a_{i_k}X_{i_k}
\end{equation}
almost surely, in which $a_{i_j}=-\lambda_{i_j}/\lambda_{i_1}$.  
The independence of $X_{i_1}, X_{i_2}, \ldots, X_{i_k}$ contradicts \eqref{eq:Contradiction}. Relation \eqref{eq:Independence} therefore has no nontrivial solutions.

\medskip\noindent\textsc{Homogeneity}. 
This follows from the bilinearity of the inner product and linearity of expectation:
\begin{equation*}
f_{\vec{X},d}(c\boldsymbol{\lambda})
=\left( \frac{\E |c\langle\vec{X}, \bm{\lambda}\rangle|^d}{\Gamma(d+1)}\right)^{1/d}
=\left( \frac{|c|^d\E |\langle\vec{X}, \bm{\lambda}\rangle|^d}{\Gamma(d+1)}\right)^{1/d}
=|c|f_{\vec{X},d}(\boldsymbol{\lambda}).
\end{equation*}

\smallskip\noindent\textsc{Triangle Inequality}. 
For $\boldsymbol{\lambda}, \boldsymbol{\mu}\in \R^n$, define random variables 
$X=\langle \vec{X},\boldsymbol{\lambda}\rangle$ and $Y=\langle \vec{X},\boldsymbol{\mu}\rangle$.
Minkowski's inequality implies
\begin{equation*}
\big(\E\vert \langle \vec{X}, \boldsymbol{\lambda}+\boldsymbol{\mu}\rangle\vert^d\big)^{1/d}
=\big(\E|X+Y|^d\big)^{1/d} \leq \big(\E|X|^d\big)^{1/d}+\big(\E|Y|^d\big)^{1/d}. 
\end{equation*} 
The triangle inequality for $f_{\vec{X},d}$ follows.
\end{proof}

\begin{proposition}\label{Proposition:Continuity}
Suppose $\vec{X}$ is an iid random vector whose entries have at least $m$ moments. The function $f:\left[1,m\right] \to \mathbb{R}$ defined by $f(d) =\norm{A}_{\vec{X},d}$ is continuous for all $A\in\h_n$.
\end{proposition}

\begin{proof} 
Define the random variable  $Y = \langle\vec{X}, \bm{\lambda}\rangle$, in which  $\boldsymbol{\lambda}$ denotes the vector of eigenvalues of $A$. The random variable $Y$ is a measurable function defined on a probability space $(\Omega, \mathcal{F}, \mathbb{P})$. The pushforward measure of $Y$ is the probability measure $\mu_{Y}$  on $\mathbb{R}$ defined by $\mu_Y(E)=\mathbb{P}\big(Y^{-1}(E)\big)$ for all Borel sets $E$. Consequently, 
\begin{equation*}
    \Gamma(d+1)\big(f(d)\big)^d = \E|Y|^d  = \int |x|^d \, d\mu_{Y}.
\end{equation*} 
The bound $|x|^d \leq |x| + |x|^m$ holds for all $x\in \R$ and $1 \leq d \leq m$. Therefore, 
\begin{equation*}
    \int |x|^d \, d\mu_{Y} \leq  \int |x|  \, d\mu_Y + \int |x|^m \, d\mu_Y.
\end{equation*} If $d_i\to d$, then $\int |x|^{d_i}d\mu_Y\to \int |x|^{d}d\mu_Y$ by the dominated convergence theorem. Consequently, 
 $ \Gamma(d_i+1)\big(f(d_i)\big)^{d_i}\to\Gamma(d+1)\big(f(d)\big)^d $ whenever $d_i\to d$. The function $\Gamma(d+1)\big(f(d)\big)^d$ is therefore continuous in $d$. The continuity of the gamma function establishes continuity for $f^d$ and $f$.
\end{proof}

\section{Remarks}\label{Section:End}

\begin{remark}
A norm $\norm{\cdot}$ on $\M_n$ is \emph{weakly unitarily invariant} if $\norm{A}=\norm{U^*AU}$ for all $A\in \M_n$ and unitary $U \in \M_n$.  A norm $\Phi$ on the space $C(S)$ of continuous functions on the unit sphere $S\subset \C^n$ is a \emph{unitarily invariant function norm} if  $\Phi(f\circ U)=\Phi(f)$ for all $f\in C(S)$ and unitary $U \in \M_n$. 
Every weakly unitarily invariant norm $\norm{\cdot}$ on $\M_n$ is of the form
$\norm{A}=\Phi(f_A)$, in which $f_A\in C(S)$ is defined by $f_A(\vec{x})=\langle A\vec{x},\vec{x}\rangle$ and $\Phi$ is a unitarily invariant function norm \cite{BhatiaHolbrook}, \cite[Thm.~2.1]{Bhatia}.
\end{remark}

\begin{remark}
Remark 3.4 of \cite{Ours} is somewhat misleading. We state there that the entries of $\vec{X}$ are required to be identically distributed but not independent. To clarify, the entries of $\vec{X}$ being identically distributed guarantees that $\norm{\cdot}_{\vec{X},d}$ satisfies the triangle inequality on $\h_n$. The additional assumption of independence guarantees that $\norm{\cdot}_{\vec{X},d}$ is also positive definite.
\end{remark}

\bibliography{RandomNormsII}
\bibliographystyle{plain}
\end{document}